\documentclass[11pt]{amsart}

\author[C.~Sanna]{Carlo Sanna$^\dagger$}
\thanks{$\dagger\,$C.~Sanna is a member of GNSAGA of INdAM and of CrypTO, the group of Cryptography and Number~Theory of Politecnico di Torino}
\address{\parbox{\linewidth}{
Politecnico di Torino, Department of Mathematical Sciences\\
Corso Duca degli Abruzzi 24, 10129 Torino, Italy\\[-8pt]}}
\email{carlo.sanna.dev@gmail.com}

\title{On the number of residues of linear recurrences}

\usepackage{amsmath}
\usepackage{amssymb}
\usepackage{amsthm}
\usepackage{cite}
\usepackage{geometry}
\usepackage{color}
\usepackage{hyperref}
\hypersetup{colorlinks=true}
\usepackage{enumitem}
\setenumerate{label=(\roman*), topsep=0.5em, itemsep=0.5em, leftmargin=3.0em}

\newtheorem{theorem}{Theorem}[section]
\newtheorem{proposition}{Proposition}[section]

\newtheorem{lemma}[theorem]{Lemma}
\theoremstyle{remark}
\newtheorem{remark}{Remark}[section]

\DeclareMathOperator{\ord}{ord}
\DeclareMathOperator{\Res}{Res}
\DeclareMathOperator{\s}{s}
\DeclareMathOperator{\sign}{sign}

\begin{document}

\begin{abstract}
For every nonconstant monic polynomial $g \in \mathbb{Z}[X]$, let $\mathfrak{M}(g)$ be the set of positive integers $m$ for which there exist an integer linear recurrence $(s_n)_{n \geq 0}$ having characteristic polynomial $g$ and a positive integer $M$ such that $(s_n)_{n \geq 0}$ has exactly $m$ distinct residues modulo $M$.
Dubickas and Novikas proved that $\mathfrak{M}(X^2 - X - 1) = \mathbb{N}$.
We study $\mathfrak{M}(g)$ in the case in which $g$ is divisible by a monic quadratic polynomial $f \in \mathbb{Z}[X]$ with roots $\alpha,\beta$ such that $\alpha\beta = \pm 1$ and $\alpha / \beta$ is not a root of unity.
We show that this problem is related to the existence of special primitive divisors of certain Lehmer sequences, and we deduce some consequences on $\mathfrak{M}(g)$.
In~particular, for $\alpha\beta = -1$, we prove that $m \in \mathfrak{M}(g)$ for every integer $m \geq 7$ with $m \neq 10$ and $4 \nmid m$.
\end{abstract}

\maketitle

\section{Introduction}

An integer sequence $\mathbf{s} = (s_n)_{n \geq 0}$ is a \emph{linear recurrence} if there exist $c_1, \dots, c_r \in \mathbb{Z}$ such that
\begin{equation}\label{equ:rec}
s_n = c_1 s_{n - 1} + c_2 s_{n - 2} + \cdots + c_r s_{n - r} ,
\end{equation}
for every integer $n \geq r$.
The values $s_0, \dots, s_{r-1}$ are the \emph{initial conditions} of $\mathbf{s}$, and
\begin{equation*}
g(X) = X^r - c_1 X^{r - 1} - c_2 X^{r - 2} - \cdots - c_r
\end{equation*}
is the \emph{characteristic polynomial} of $\mathbf{s}$.
Together they completely determine $\mathbf{s}$ via~\eqref{equ:rec}.
A classic example of linear recurrence is the sequence of \emph{Fibonacci numbers}, having initial conditions $0, 1$ and characteristic polynomial $X^2 - X - 1$.
It is easily seen that $\mathbf{s}$ is ultimately periodic modulo $M$, for every positive integer $M$, and purely periodic if $(c_r, M) = 1$.
Indeed, properties of linear recurrences modulo $M$ have been studied intensively, including: which residues modulo $M$ appear in the $\mathbf{s}$ and how frequently~\cite{MR369240, MR3298566, MR2024599, MR1119645, MR1117027, MR1393478}, and for which positive integers $M$ the linear recurrence $\mathbf{s}$ contains a complete system of residues modulo $M$~\cite{MR3068231, MR294238, MR951911, MR3696266}.

Let $\mathfrak{M}(g)$ denote the set of positive integers $m$ such that there exist initial conditions $s_0, \dots, s_{r - 1} \in \mathbb{Z}$ and a positive integer $M$ for which the linear recurrence $\mathbf{s}$ has exactly $m$ distinct residues modulo~$M$.
Dubickas and Novikas~\cite{MR4125906} proved that $\mathfrak{M}(X^2 - X - 1) = \mathbb{N}$ and stated that the problem of determining $\mathfrak{M}(g)$ ``may be very diﬃcult in general''.
The first step of their proof is a lemma regarding roots of $X^2 - X - 1$ modulo $p$ that have a prescribed multiplicative order~\cite[Lemma~3]{MR4125906}.
We provide below a straighforward generalization of it. (The proof is postponed to Section~\ref{sec:prel}).

\begin{lemma}\label{lem:criterion}
Let $f, g \in \mathbb{Z}[x]$ be nonconstant monic polynomials with $f \mid g$, let $m$ be a positive integer, and let $p$ be a prime number.
Suppose that:
\begin{enumerate}
\item\label{ite:a} There exists $a \in \mathbb{Z}$ such that $p \mid f(a)$ and $\ord_p(a) = m$.
\end{enumerate}
Then $m \in \mathfrak{M}(g)$.
\end{lemma}

For $f = g = X^2 - X - 1$ and for positive integers $m$ belonging to certain residue classes modulo $40$, Dubickas and Novikas showed how to construct $a$ and $p$ satisfying~\ref{ite:a} by using primitive divisors of Lucas numbers~\cite[Lemma~5, Lemma~7]{MR4125906}.

Our first contribution is the next theorem, which shows that for more general quadratic polynomials $f$ the statement \ref{ite:a} is equivalent to $p$ being a particular primitive divisor of a certain term of a Lehmer sequence.

Let $\gamma, \delta$ be complex numbers such that $\gamma\delta$ and $(\gamma + \delta)^2$ are nonzero coprime integers and $\gamma / \delta$ is not a root of unity.
The \emph{Lehmer sequence} $(u_n(\gamma, \delta))_{n \geq 0}$ associated to $\gamma, \delta$ is defined by
\begin{equation*}
u_n(\gamma, \delta) := 
\begin{cases}
(\gamma^n - \delta^n) / (\gamma - \delta) & \text{ if $2 \nmid n$} , \\
(\gamma^n - \delta^n) / (\gamma^2 - \delta^2) & \text{ if $2 \mid n$} ,
\end{cases}
\end{equation*}
for all integers $n \geq 0$.
The conditions on $\gamma, \delta$ ensure that each $u_n(\gamma, \delta)$ is an integer.
A prime number $p$ is a \emph{primitive divisor} of $u_n(\gamma, \delta)$ if $p \mid u_n(\gamma, \delta)$ but $p \nmid (\gamma^2 - \delta^2)^2 u_1(\gamma, \delta) \cdots u_{n-1}(\gamma, \delta)$.

\begin{theorem}\label{thm:alphabeta}
Let $f \in \mathbb{Z}[X]$ be a monic quadratic polynomial with roots $\alpha, \beta$ such that \mbox{$\alpha \beta = \pm 1$} and $\alpha / \beta$ is not a root of unity.
Also, let $m$ be a positive integer and let $p$ be a prime number.
If $\alpha\beta = -1$ then put $\gamma := \alpha$, $\delta := -\beta$, and $n := m / (m, 2)$, while if $\alpha\beta = +1$ then put $\gamma := \alpha^{1/2}$, $\delta := \alpha^{-1/2}$, and $n := m$.
Then~\ref{ite:a} is equivalent to:
\begin{enumerate}
\setcounter{enumi}{1}
\item\label{ite:prim} $p$ is a primitive divisor of $u_n(\gamma, \delta)$ and $p \equiv 1 \pmod m$.
\end{enumerate}
Moreover, each of the following implies~\ref{ite:a} and~\ref{ite:prim}:
\begin{enumerate}
\setcounter{enumi}{2}
\item\label{ite:suff1} $\alpha\beta = -1$, $4 \nmid m$, $m \notin \{3, 6\}$, and $p$ is a primitive divisor of $u_{m/(m,2)}(\gamma, \delta)$.
\item\label{ite:suff2} $\alpha\beta = -1$, $8 \mid m$, $p$ is a primitive divisor of $u_{m / 2}(\gamma, \delta)$, and $p \equiv 1 \pmod 4$.
\item\label{ite:suff3} $\alpha\beta = +1$, $4 \mid m$, $p$ is a primitive divisor of $u_m(\gamma, \delta)$, and $p \equiv 1 \pmod 4$.
\end{enumerate}
\end{theorem}

As consequences of Theorem~\ref{thm:alphabeta}, Lemma~\ref{lem:criterion}, and results on the existence of primitive divisors of terms of Lehmer sequences (Lemma~\ref{lem:gammadelta1} and~\ref{lem:1mod4} below), we obtain the following results on $\mathfrak{M}(g)$.

\begin{theorem}\label{thm:odd}
Let $f, \alpha, \beta$ be as in Theorem~\ref{thm:alphabeta} with $\alpha\beta = -1$, and let $g \in \mathbb{Z}[X]$ be a monic polynomial with $f \mid g$.
Then $m \in \mathfrak{M}(g)$ for every integer $m \geq 7$, with $m \neq 10$ and $4 \nmid m$.
\end{theorem}

\begin{theorem}\label{thm:specialD0}
Let $f, \alpha, \beta$ be as in Theorem~\ref{thm:alphabeta} and let $g \in \mathbb{Z}[X]$ be a monic polynomial with $f \mid g$.
Write $(\alpha - \beta)^2 = D_0 D_1^2$, where $D_0, D_1 \in \mathbb{Z}$ and $D_0$ is squarefree.
Suppose that $D_0 \geq 5$ and $D_0 \equiv 1 \pmod 4$.
Then $m \in \mathfrak{M}(g)$ for every positive integer $m$ with $8D_0 \mid m$ if $\alpha\beta = -1$, and $4D_0 \mid m$ if $\alpha\beta = +1$.
\end{theorem}

Given two specific polynomials $f, g \in \mathbb{Z}[X]$ satisfying the hypothesis of Theorem~\ref{thm:odd} and Theorem~\ref{thm:specialD0}, one can try to determine $\mathfrak{M}(g)$ by using the aforementioned theorems and by employing~\cite[Lemma~6]{MR4125906}.
However, this requires a meticulous inspection of the numerical values of certain linear recurrences of characteristic polynomial $g$, and a detailed case-by-case analysis, as the one done by Dubickas and Novikas for $\mathfrak{M}(X^2 - X - 1)$~\cite[Sections~6--8]{MR4125906}.

\begin{remark}
It should be possible to provide an equivalent version of Theorem~\ref{thm:alphabeta} in terms of primitive divisors of the \emph{Lehmer--Pierce sequence} $(\Delta_n(\alpha, \beta))_{n \geq 0}$~\cite{MR2468478, MR1503584, MR1503118}, which is defined by $\Delta_n(\alpha, \beta) := (\alpha^n - 1)(\beta^n - 1)$ for every integer $n \geq 0$.
\end{remark}

\section{Notation}

For every integer $a$ and for each prime number $p$, we let $\ord_p(a)$ denote the multiplicative order of $a$ modulo $p$, with the implicit condition that $p \nmid a$.
Also, when $p$ is odd, we write $\big(\tfrac{a}{p}\big)$ for the Legendre symbol.
For algebraic integers $\zeta$ and $\eta$, the notation $\zeta \equiv \eta \pmod p$ means that $p$ divides $\zeta - \eta$, that is, $(\zeta - \eta) / p$ is an algebraic integer.
For every positive integer~$n$, we let $\varphi(n)$ be the Euler totient function of $n$.
Furthermore, we write $\Phi_n(X)$ for the $n$th cyclotomic polynomial, and $\Phi_n(X, Y) := \Phi_n(X/Y)Y^{\varphi(n)}$ for its homogenization.
Given two monic polynomials $f, g \in \mathbb{Z}[X]$, we let $\Res(f, g)$ denote their resultant.

\section{Preliminaries}\label{sec:prel}

We begin by proving Lemma~\ref{lem:criterion}.

\begin{proof}[Proof of Lemma~\ref{lem:criterion}]
Put $r := \deg(g)$ and let $\mathbf{s} = (s_n)_{n \geq 0}$ be the linear recurrence with initial conditions $1, a, \dots, a^{r-1}$ and characteristic polynomial $g$.
We shall prove that $s_n \equiv a^n \pmod p$ for every integer $n \geq 0$.
In~turn, since $\ord_p(a) = m$, this implies that $\mathbf{s}$ has exactly $m$ distinct residues modulo $p$, namely $1, a, \dots, a^{m-1} \!\!\pmod p$, and consequently $m \in \mathfrak{M}(g)$.
Let us proceed by induction on $n$.
For $n = 0,\dots,r-1$ the claim is obvious because of the initial conditions of $\mathbf{s}$.
Assuming that the claim is true for every nonnegative integer less than $n$, let us prove it for $n$.
From~\eqref{equ:rec} and the induction hypothesis, we have that
\begin{align*}
s_{n} &\equiv c_1 s_{n - 1} + c_2 s_{n - 2} + \cdots + c_r s_{n - r} \equiv c_1 a^{n - 1} + c_2 a^{n - 2} + \cdots + c_r a^{n - r} \\
    &\equiv a^{n - r} (a^r - g(a)) \equiv a^n \pmod p ,
\end{align*}
because $p \mid f(a) \mid g(a)$.
\end{proof}

The next result is a simple equivalence for~\ref{ite:a}.

\begin{lemma}\label{lem:res}
Let $f \in \mathbb{Z}[X]$ be a nonconstant monic polynomial, let $m$ be a positive integer, and let $p$ be a prime number.
Then \ref{ite:a} is equivalent to $p \mid \Res(f, \Phi_m)$ and $p \equiv 1 \pmod m$.
\end{lemma}
\begin{proof}
On the one hand, if~\ref{ite:a} holds then $a$ is a primitive $m$th root of unity modulo $p$.
Hence, $p \equiv 1 \pmod m$ and $a$ is a root of $\Phi_m$ modulo $p$.
Since $p \mid f(a)$, we have that $a$ is a common root of $f$ and $\Phi_m$ modulo $p$, and consequently $p \mid \Res(f, \Phi_m)$.
On the other hand, if $p \mid \Res(f, \Phi_m)$ and $p \equiv 1 \pmod m$ then $\Phi_m$ splits completely modulo $p$ and it has a common root with $f$, thus~\ref{ite:a} follows.
\end{proof}

We need some results on Lehmer sequences and related values of cyclotomic polynomials.
It~is known that a prime number $p$ divides some term of a Lehmer sequence $(u_n(\gamma, \delta))_{n \geq 0}$ if and only if $p \nmid \gamma\delta$.
In such a case, let $r_p(\gamma, \delta)$ be the \emph{rank of appearance} of $p$, that is, the smallest positive integer $k$ such that $p \mid u_k(\gamma, \delta)$.
Furthermore, it can be proved that $\Phi_n(\gamma, \delta) \in \mathbb{Z}$ for every integer $n \geq 3$ (for these facts see, e.g.,~\cite{MR491445}).

\begin{lemma}\label{lem:lehmer}
Let $(u_k(\gamma, \delta))_{k \geq 0}$ be a Lehmer sequence, let $p$ be a prime number, and let $n \geq 3$ be an integer.
Then we have the following:
\begin{enumerate}[label=(p{\small\arabic*})]
\item\label{ite:div} $p \mid u_n(\gamma, \delta)$ if and only if $p \nmid \gamma\delta$ and $r_p(\gamma, \delta) \mid n$.
\item\label{ite:p2p} If $p \nmid \gamma\delta$ and $p \mid (\gamma^2 - \delta^2)^2$ then $r_p(\gamma, \delta) \in \{p, 2p\}$.
\item\label{ite:r2} If $2 \nmid \gamma\delta(\gamma^2 - \delta^2)^2$ then $r_2(\gamma, \delta) = 3$.
\item\label{ite:leg} If $p \nmid 2\gamma\delta(\gamma + \delta)^2$ then $p \equiv \Big(\tfrac{(\gamma^2 - \delta^2)^2}{p}\Big) \pmod {r_p(\gamma, \delta)}$.
\item\label{ite:Phin_factor} If $p \mid \Phi_n(\gamma, \delta)$ then $p \nmid \gamma\delta$ and $n = r_p(\gamma, \delta) p^v$ for some integer $v \geq 0$.
\item\label{ite:prim_equiv} $p$ is a primitive divisor of $u_n(\gamma, \delta)$ if and only if $p \mid \Phi_n(\gamma, \delta)$ and $p \equiv \pm 1 \pmod n$.
\item\label{ite:prim_odd} If $n \geq 5$, $2 \nmid n$, and $p$ is a primitive divisor of $u_n(\gamma, \delta)$ then $p \equiv \Big(\tfrac{\gamma\delta(\gamma - \delta)^2}{p}\Big) \pmod {2n}$.
\item\label{ite:prim_4} If $4 \mid n$, $\gamma\delta = 1$, $\gamma - \delta \in \mathbb{Z}$, $p$ is a primitive divisor of $u_n(\gamma, \delta)$, and $p \equiv 1 \pmod 4$ then $p \equiv 1 \pmod {2n}$.
\end{enumerate}
\end{lemma}
\begin{proof}
Properties~\ref{ite:div}, \ref{ite:p2p}, and~\ref{ite:r2} follow from~\cite[Corollary~2.2]{MR1863855}, \ref{ite:leg} is~\cite[Theorem~1.9]{MR1502953}, and~\ref{ite:Phin_factor} is~\cite[Proposition~2.3]{MR1863855}.

Let us prove~\ref{ite:prim_equiv}.
On the one hand, if $p$ is a primitive divisor of $u_n(\gamma, \delta)$ then $p \mid u_n(\gamma, \delta)$ but $p \nmid (\gamma^2 - \delta^2)^2 u_1(\gamma, \delta) \cdots u_{n - 1}(\gamma, \delta)$.
Since for every positive integer $k$ we have that
\begin{equation*}
u_k(\gamma, \delta) = 
\begin{cases}
\prod_{d \,\mid\, k,\, d \,>\, 1} \Phi_d(\gamma, \delta) & \text{ if $2 \nmid k$}, \\
\prod_{d \,\mid\, k,\, d \,>\, 2} \Phi_d(\gamma, \delta) & \text{ if $2 \mid k$}, 
\end{cases}
\end{equation*}
it follows that $p \mid \Phi_n(\gamma, \delta)$.
Moreover, $n = r_p(\gamma, \delta)$ and, also by \ref{ite:div}, $p \nmid \gamma\delta(\gamma^2 - \delta^2)^2$.
Hence, from~\ref{ite:r2} and~\ref{ite:leg} we get that $p \equiv \pm 1 \pmod n$.
On the other hand, if $p \mid \Phi_n(\gamma, \delta)$ and $p \equiv \pm 1 \pmod n$ then from~\ref{ite:Phin_factor} we get that $p \nmid \gamma\delta$ and $n = r_p(\gamma, \delta)$.
Hence, $p \mid u_n(\gamma, \delta)$ but $p \nmid u_1(\gamma, \delta) \cdots u_{n - 1}(\gamma, \delta)$.
Also,~\ref{ite:p2p} yields that $p \nmid (\gamma^2 - \delta^2)^2$.
Hence, $p$ is a primitive divisor of $u_n(\gamma, \delta)$.

Now let us prove~\ref{ite:prim_odd}.
Since $p$ is a primitive divisor of $u_n(\gamma, \delta)$, we have that $n = r_p(\gamma, \delta)$ and, also by~\ref{ite:div}, $p \nmid \gamma\delta(\gamma^2 - \delta^2)^2$.
By $n \geq 5$ and~\ref{ite:r2}, we get that $p > 2$.
Since $2 \nmid n$, we have that $v_n(\gamma, \delta) := (\gamma^n + \delta^n) / (\gamma + \delta)$ is an integer.
Moreover, from $p \mid u_n(\gamma, \delta)$ and the identity
\begin{equation*}
(\gamma + \delta)^2 v_n(\gamma, \delta)^2 - (\gamma - \delta)^2 u_n(\gamma, \delta)^2 = 4 (\gamma\delta)^n ,
\end{equation*}
it follows that $(\gamma + \delta)^2 v_n(\gamma, \delta)^2 \equiv 2^2 (\gamma\delta)^n \pmod p$.
Hence, $\Big(\tfrac{(\gamma + \delta)^2}{p}\Big) = \Big(\tfrac{\gamma\delta}{p}\Big)$ and consequently $\Big(\tfrac{(\gamma^2 - \delta^2)^2}{p}\Big) = \Big(\tfrac{\gamma\delta(\gamma - \delta)^2}{p}\Big)$.
Then by~\ref{ite:leg} we obtain that $p \equiv \Big(\tfrac{\gamma\delta(\gamma - \delta)^2}{p}\Big) \pmod n$.
Recalling that $p$ and $n$ are both odd, it follows that $p \equiv \Big(\tfrac{\gamma\delta(\gamma - \delta)^2}{p}\Big) \pmod {2n}$.

It remains to prove~\ref{ite:prim_4}.
Since $p$ is a primitive divisor of $u_n(\gamma, \delta)$, we have that $n = r_p(\gamma, \delta)$ and, also by~\ref{ite:div}, $p \nmid \gamma\delta(\gamma^2 - \delta^2)^2$.
From $4 \mid n$, $p \equiv 1 \pmod 4$, and~\ref{ite:leg}, it follows that $\Big(\tfrac{(\gamma^2 - \delta^2)^2}{p}\Big) = 1$.
Also, $(\gamma^2 - \delta^2)^2 = (\gamma - \delta)^2 (\gamma + \delta)^2$ and $\gamma - \delta$ is an integer.
Hence, $\Big(\tfrac{(\gamma + \delta)^2}{p}\Big) = 1$.
Noting that $(\gamma + \delta)^2$ is the discriminant of $(X - \gamma)(X + \delta) = X^2 - (\gamma - \delta)X - \gamma\delta \in \mathbb{Z}[X]$, one gets that $\gamma^{p-1} \equiv 1 \pmod p$.
Multiplying both sides by $\delta^{(p-1)/2}$, and recalling that $\gamma\delta = 1$, it follows that $\gamma^{(p-1)/2} \equiv \delta^{(p-1)/2} \pmod p$, and so $p \mid u_{(p-1)/2}(\gamma, \delta)$.
Consequently, by~\ref{ite:div}, we have that $n \mid (p - 1) / 2$, that is, $p \equiv 1 \pmod {2n}$.
\end{proof}

We also need the following identity for a product of cyclotomic polynomials.

\begin{lemma}\label{lem:refle}
For every positive integer $m$, we have
\begin{equation*}
\Phi_m(X) \Phi_m(-X) = (-1)^{\varphi(m)} \Phi_{m / \!\!\;(m, 2)}\!\big(X^2\big)^e ,
\end{equation*}
where $e := 1$ if $4 \nmid m$, and $e := 2$ if $4 \mid m$.
\end{lemma}
\begin{proof}
For every positive integer $n$, let $\zeta_n := \mathrm{e}^{2 \pi \mathbf{i}/n}$ be a primitive $n$th root of unity.
We have
\begin{equation}\label{equ:refle1}
\Phi_m(X) \Phi_m(-X) = \prod_{\substack{1 \,\leq\, k \,\leq\, m \\ (k,\, m) \,=\, 1}} (X - \zeta_m^k) (-X - \zeta_m^k) = (-1)^{\varphi(m)} \prod_{\substack{1 \,\leq\, k \,\leq\, m \\ (k,\, m) \,=\, 1}} (X^2 - \zeta_m^{2k}) .
\end{equation}
If $2 \nmid m$ then $\zeta_m^2$ is a primitive $m$th root of unity and the last product of~\eqref{equ:refle1} is equal to $\Phi_m\big(X^2\big)$.
If $2 \mid m$ then $\zeta_m^2 = \zeta_{m / 2}$ is a primitive $(m/2)$th root of unity.
Also, if $2 \mid\mid m$ then $\zeta_{m / 2}^2$ is a primitive $(m/2)$th root of unity, and the last product of~\eqref{equ:refle1} is equal to
\begin{align*}
\prod_{\substack{1 \,\leq\, k \,\leq\, m \\ (k,\, m) \,=\, 1}} (X^2 - \zeta_{m/2}^k) &= \prod_{\substack{1 \,\leq\, k \,\leq\, m \\ (k,\, m/2) \,=\, 1}} (X^2 - \zeta_{m/2}^k) \prod_{\substack{1 \,\leq\, h \,\leq\, m/2 \\ (h,\, m/2) \,=\, 1}} (X^2 - \zeta_{m/2}^{2h})^{-1} \\
    &= \Phi_{m/2}\big(X^2\big)^2 / \Phi_{m/2}\big(X^2\big) = \Phi_{m/2}\big(X^2\big) .
\end{align*}
If $4 \mid m$ then the last product of~\eqref{equ:refle1} is equal to
\begin{equation*}
\prod_{\substack{1 \,\leq\, k \,\leq\, m \\ (k,\, m) \,=\, 1}} (X^2 - \zeta_{m/2}^k) = \prod_{\substack{1 \,\leq\, k \,\leq\, m \\ (k,\, m / 2) \,=\, 1}} (X^2 - \zeta_{m/2}^k) = \Phi_{m / 2}\big(X^2\big)^2 ,
\end{equation*}
and the proof is complete.
\end{proof}

The problem of determining which terms of a Lehmer sequence have a primitive divisor has a very long history.
The first complete classification was given by Bilu, Hanrot, and Voutier~\cite{MR1863855} (see also~\cite{MR2289425}).
We make use of the following particular case.

\begin{lemma}\label{lem:gammadelta1}
Let $(u_k(\gamma, \delta))_{k \geq 0}$ be a Lehmer sequence with $\gamma\delta = 1$.
Then $u_n(\gamma, \delta)$ has a primitive divisor for every positive integer $n \notin \{1, 2, 3, 4, 5, 6, 10, 12\}$.
\end{lemma}
\begin{proof}
Following~\cite{MR1863855}, we can write $\gamma = \zeta \big(\!\sqrt{a} - \sqrt{b}\big) / 2$ and $\delta = \zeta \big(\!\sqrt{a} + \sqrt{b}\big) / 2$, where $a, b$ are integers and $\zeta$ is a fourth root of unity.
In~particular, $\gamma\delta = 1$ implies that $a - b = \pm 4$.
Let $n \geq 3$ be an integer and suppose that $u_n(\gamma, \delta)$ has no primitive divisor.
By~\cite[Theorem~1.4]{MR1863855}, we have that $n \leq 30$.
If $7 \leq n \leq 30$ and $n \notin \{8, 10, 12\}$, then by~\cite[Theorem~C]{MR1863855} we have that $(a, b)$ belongs to~\cite[Table~2]{MR1863855}, but none of the pairs in such table satisfies $a - b = \pm 4$.
If $n \in \{3,4,5,6,8,10,12\}$ then by~\cite[Theorem~1.3]{MR1863855} we have that $(a, b)$ belongs to~\cite[Table~4]{MR1863855} and, checking again the condition $a - b = \pm 4$, we get that $n \in \{3, 4, 5, 6, 10, 12\}$ (see Remark~\ref{rem:5def}).
\end{proof}

\begin{remark}\label{rem:5def}
In line $n = 5$ of~\cite[Table~4]{MR1863855}, one has to include also the pair $(-1, -5)$, which is $(\psi_{k - 2\varepsilon}, \psi_{k - 2\varepsilon} - 4\psi_k)$ for $k = 1$ and $\varepsilon = 1$ (note that $\psi_{-1} = -1$).
This is lost when in~\cite[p.~89]{MR1863855} it is claimed that ``By (28), we have [...] $k \neq 1$ in the case (35)'' but $k = 1$ (and $\varepsilon = 1$) does not contradict~\cite[Eq.~(28)]{MR1863855}.
Similarly, in line $n = 10$ of~\cite[Table~4]{MR1863855}, one has to include also the pair $(-5, -1)$, which is $(\psi_{k - 2\varepsilon} - 4\psi_k, \psi_{k - 2\varepsilon})$ for $k = 1$ and $\varepsilon = 1$.
\end{remark}

\begin{remark}
Lemma~\ref{lem:gammadelta1} cannot be improved without further information on $\gamma, \delta$.
Indeed, it can be checked that $u_n\!\Big(\tfrac{\sqrt{5} - 1}{2}, \tfrac{\sqrt{5} + 1}{2}\Big)$ for $n \in \{1,2,6,10,12\}$, 
$u_n\!\Big(\tfrac{\sqrt{-2} - \sqrt{-6}}{2}, \tfrac{\sqrt{-2} + \sqrt{-6}}{2}\Big)$ for $n \in \{3,4\}$, 
and $u_5\!\Big(\tfrac{\sqrt{-1} - \sqrt{-5}}{2}, \tfrac{\sqrt{-1} + \sqrt{-5}}{2}\Big)$ have no primitive divisor.
\end{remark}

We need the identities for the \emph{Aurifeuillian factorizations} of the cyclotomic polynomials~\cite{MR922449}.
However, instead of using them how it is commonly done, that is, to write values of the cyclotomic polynomials as differences of two squares and thus factorize them; we use them to write values of the cyclotomic polynomials as sums of two squares (proof of Lemma~\ref{lem:1mod4} below).

A polynomial $F \in \mathbb{Z}[X, Y]$ is \emph{symmetric}, respectively \emph{antisymmetric}, if $F(Y, X) = F(X, Y)$, respectively $F(Y, X) = -F(X, Y)$.
The \emph{symmetry type} of $F$ is $\s(F) = +1$ if $F$ is symmetric, and $\s(F) = -1$ if $F$ is antisymmetric.

\begin{lemma}\label{lem:auri}
Let $k$ be a squarefree integer and let $n \geq 3$ be an integer.
Suppose that one of the following conditions holds:
\begin{enumerate}[label=(c{\small\arabic*})]
\item\label{ite:c1} $k \equiv 1 \pmod 4$, $k \mid n$, and $2k \nmid n$.
\item\label{ite:c2} $k \not\equiv 1 \pmod 4$, $2k \mid n$, and $4k \nmid n$.
\end{enumerate}
Then there exist homogeneous polynomials $F_{n, k}, G_{n, k} \in \mathbb{Z}[X, Y]$ such that
\begin{equation*}
\Phi_n(X, Y) = F_{n, k}(X, Y)^2 - k (XY)^{q_n} G_{n, k}(X, Y)^2 ,
\end{equation*}
where $q_n := \prod_{p \,>\, 2,\; p^v \,\mid\mid\, n} p^{v - 1}$.
Furthermore, we have
\begin{equation*}
\deg(F_{n, k}) = \frac{\varphi(n)}{2}, \quad \deg(G_{n, k}) = \frac{\varphi(n)}{2} - q_n ,
\end{equation*}
while
\begin{equation*}
\s(F_{n, k}) = 
    \begin{cases}
    1 & \text{ if $k = 1$, or $k > 1$ and $2 \mid n$}, \\
    (-1)^{\varphi(n) / 2} & \text{ otherwise},
    \end{cases}
\end{equation*}
and $\s(G_{n, k}) = \sign(k) \s(F_{n, k})$.
\end{lemma}
\begin{proof}
The claim is the homogeneous version of~\cite[Theorem~2.1]{MR922449}.
\end{proof}

\begin{lemma}\label{lem:auri_coprime}
Let $n$ and $k$ be as in Lemma~\ref{lem:auri}, and let $\zeta, \eta$ be algebraic integers.
If a prime number $p$ divides both $F_{n,k}(\zeta, \eta)$ and $G_{n,k}(\zeta, \eta)$ then $p$ divides $2n(\zeta \eta)^j$ for some integer $j \geq 0$.
(Recall that we say that $p$ divides an algebraic integer $\xi$ if $\xi / p$ is an algebraic integer.)
\end{lemma}
\begin{proof}
With the notation of Lemma~\ref{lem:auri}, we can write $n = q_n m$ for an integer $m \geq 3$ with $q_m = 1$ and such that the hypothesis of Lemma~\ref{lem:auri} holds with $m$ in place of $n$.
Moreover, by~\cite[Eqs.~(2)]{MR922449} we have that $F_{n,k}(X, Y) = F_{m,k}(X^{q_n}, Y^{q_n})$ and $G_{n,k}(X, Y) = G_{m,k}(X^{q_n}, Y^{q_n})$.
Therefore, without loss of generality, we can assume that $q_n = 1$.

Let $i_{n,k}$ be the order of $\mathbb{Z}[X, \sqrt{kX}] / I$, where $I$ is the ideal generated by 
\begin{equation}\label{equ:FplusminusG}
F_{n,k}(X, 1) - G_{n,k}(X, 1)\sqrt{kX} \quad \text{ and } \quad F_{n,k}(X, 1) + G_{n,k}(X, 1)\sqrt{kX}
\end{equation}
in $\mathbb{Z}[X, \sqrt{kX}]$.
Then $i_{n,k}$ is a linear combination of~\eqref{equ:FplusminusG} in $\mathbb{Z}[X, \sqrt{kX}]$, and by homogeneization $i_{n,k} Y^j$ is a linear combination of
\begin{equation*}
F_{n,k}(X) - G_{n,k}(X)\sqrt{kXY} \quad \text{ and } \quad F_{n,k}(X, Y) + G_{n,k}(X, Y)\sqrt{kXY} ,
\end{equation*}
in $\mathbb{Z}[X, Y, \sqrt{kXY}]$, for some integer $j \geq 0$.
Substituting $X = \zeta$ and $Y = \eta$, we get that if $p$ divides both $F_{n,k}(\zeta, \eta)$ and $G_{n,k}(\zeta, \eta)$ then it divides $i_{n,k} \eta^j$, and so it divides $i_{n,k} (\zeta\eta)^j$.
From~\cite[Lemma~2.6]{MR922449} (which requires $q_n = 1$) we have that $i_{n,k}$ divides $(8n)^{\varphi(n)}$, and thus the claim follows.
\end{proof}

Now we can prove a result on primitive divisors of Lehmer sequences.

\begin{lemma}\label{lem:1mod4}
Let $(u_k(\gamma, \delta))_{k \geq 0}$ be a Lehmer sequence and write $(\gamma^2 - \delta^2)^2 = D_0 D_1^2$ where $D_0, D_1 \in \mathbb{Z}$ and $D_0$ is squarefree.
Suppose that $D_0 \geq 5$ and $D_0 \equiv 1 \pmod 4$.
Then, for every positive integer $\ell$ such that $4D_0 \mid \ell$, we have that each odd primitive divisor $p$ of $u_\ell(\gamma, \delta)$ satisfies $p \equiv 1 \pmod 4$.
\end{lemma}
\begin{proof}
Since $4D_0 \mid \ell$, we can write $\ell = 2^v n$ for some positive integers $v$ and $n$ with $2D_0 \mid n$ and $4D_0 \nmid n$.
Put $k := -D_0$.
By the hypotheses on $D_0$, we have that $k$ is negative and squarefree, $n \geq 3$, and~\ref{ite:c2} holds.
Moreover, since $D_0$ is squarefree and $D_0 \equiv 1 \pmod 4$, it follows that $4 \mid \varphi(D_0)$, and so $4 \mid \varphi(n)$.
Therefore, by Lemma~\ref{lem:auri}, we get that
\begin{equation}\label{equ:1mod4_eq1}
\Phi_n(X, Y) = F_{n, k}(X, Y)^2 + D_0 (XY)^{q_n} G_{n, k}(X, Y)^2 ,
\end{equation}
for some homogeneous polynomials $F_{n, k}, G_{n, k} \in \mathbb{Z}[X, Y]$, with $F_{n, k}$ symmetric and $G_{n, k}$ antisymmetric.
Since $G_{n,k}$ is antisymmetric and homogeneous, we have that $G_{n,k}(X, Y) = (X - Y)H_{n, k}(X, Y)$ for some symmetric homogeneous polynomial $H_{n, k} \in \mathbb{Z}[X, Y]$.
Now $F_{n, k}\big(X^{2^v}, Y^{2^v}\big)$ and $H_{n, k}\big(X^{2^v}, Y^{2^v}\big)$ are both symmetric homogeneous polynomials of even degree, and thus they are polynomials in $XY$ and $(X + Y)^2$ with integer coefficients.
Then, recalling that $\gamma\delta$ and $(\gamma + \delta)^2$ are integers, it follows that $F_{n, k}\big(\gamma^{2^v}, \delta^{2^v}\big)$ and $H_{n, k}\big(\gamma^{2^v}, \delta^{2^v}\big)$ are integers.
Since $2 \mid n$, by~\eqref{equ:1mod4_eq1} we get that
\begin{align}\label{equ:1mod4_eq2}
\Phi_\ell(\gamma, \delta) &= \Phi_{2^v n}(\gamma, \delta) = \Phi_{n}\big(\gamma^{2^v}, \delta^{2^v}\big) = F_{n, k}\big(\gamma^{2^v}, \delta^{2^v}\big)^2 + D_0 (\gamma\delta)^{2^v q_n} G_{n, k}\big(\gamma^{2^v}, \delta^{2^v}\big)^2 \\
    &= F_{n, k}\big(\gamma^{2^v}, \delta^{2^v}\big)^2 + D_0 (\gamma\delta)^{2^v q_n} \big(\gamma^{2^v} - \delta^{2^v}\big)^2 H_{n, k}\big(\gamma^{2^v}, \delta^{2^v}\big)^2 \nonumber \\
    &= F_{n, k}\big(\gamma^{2^v}, \delta^{2^v}\big)^2 + D_0 (\gamma\delta)^{2^v q_n} (\gamma^2 - \delta^2)^2 u_{2^v}(\gamma, \delta)^2 H_{n, k}\big(\gamma^{2^v}, \delta^{2^v}\big)^2 \nonumber \\
    &= A^2 + B^2 , \nonumber
\end{align}
where $A := F_{n, k}\big(\gamma^{2^v}, \delta^{2^v}\big)$ and $B := D_0 D_1 (\gamma\delta)^{2^{v-1} q_n} u_{2^v}(\gamma, \delta) H_{n, k}\big(\gamma^{2^v}, \delta^{2^v}\big)$ are both integers.
Let $p$ be an odd primitive divisor of $u_\ell(\gamma, \delta)$.
Hence, also by Lemma~\ref{lem:lehmer}\ref{ite:div} and \ref{ite:prim_equiv}, we have that $p \nmid \gamma\delta(\gamma^2 - \delta^2)^2 n$ and $p \mid \Phi_\ell(\gamma, \delta)$.
Thus, from~\eqref{equ:1mod4_eq2} and Lemma~\ref{lem:auri_coprime}, it follows that $p \mid A^2 + B^2$ but $p \nmid A$ and $p \nmid B$.
Consequently, we have that $p \equiv 1 \pmod 4$.
\end{proof}

\section{Proof of Theorem~\ref{thm:alphabeta}}

Let us begin by proving the equivalence of~\ref{ite:a} and~\ref{ite:prim}.
Let $D := (\alpha - \beta)^2$ be the discriminant of $f$.
First, assume that $\alpha\beta = -1$.
Note that $\gamma\delta = 1$ and $(\gamma + \delta)^2 = D$ are nonzero coprime integers and $\gamma / \delta = -\alpha/\beta$ is not a root of unity, so that $(u_k(\gamma, \delta))_{k \geq 0}$ is a Lehmer sequence.
Put $R_m^{(\varepsilon)} := \Res(f(\varepsilon X), \Phi_m(X))$ for $\varepsilon \in \{-1, +1\}$.
The roots of $f(-X)$ are $-\alpha$ and $-\beta$, while $\gamma / \delta = \alpha^2$ and $\delta / \gamma = \beta^2$.
Hence, from Lemma~\ref{lem:refle}, it follows that
\begin{align}\label{equ:RR}
R_m^+ R_m^- &= \Phi_m(\alpha) \Phi_m(\beta) \Phi_m(-\alpha) \Phi_m(-\beta) = \Phi_m(\alpha) \Phi_m(-\alpha) \Phi_m(\beta) \Phi_m(-\beta) \\
    &= \left(\Phi_n\big(\alpha^2\big) \Phi_n\big(\beta^2\big)\right)^e = \big(\Phi_n(\gamma / \delta) \Phi_n(\delta / \gamma) \big)^e = \Big(\Phi_n(\gamma, \delta) \delta^{-\varphi(n)} \Phi_n(\delta, \gamma) \gamma^{-\varphi(n)} \Big)^e \nonumber \\
    &= \big(\Phi_n(\gamma, \delta) \Phi_n(\delta, \gamma)\big)^e = \pm \Phi_n(\gamma, \delta)^{2e} , \nonumber
\end{align}
where $e := 1$ if $4 \nmid m$, and $e := 2$ if $4 \mid m$.

Suppose that~\ref{ite:a} holds.
By Lemma~\ref{lem:res}, we have that $p \mid R_m^+$ and $p \equiv 1 \pmod m$.
Hence, from~\eqref{equ:RR} and the fact that $n \mid m$, we get that $p \mid \Phi_n(\gamma, \delta)$ and $p \equiv 1 \pmod n$.
Therefore, Lemma~\ref{lem:lehmer}\ref{ite:prim_equiv} implies that $p$ is a primitive divisor of $u_n(\gamma, \delta)$, and~\ref{ite:prim} follows.

Now suppose that~\ref{ite:prim} holds.
Thus, from Lemma~\ref{lem:lehmer}\ref{ite:prim_equiv}, it follows that $p \mid \Phi_n(\gamma, \delta)$.
Consequently, by~\eqref{equ:RR}, we get that either $p \mid R_m^+$ or $p \mid R_m^-$.
In the first case,~\ref{ite:a} follows immediately from Lemma~\ref{lem:res}.
In the second case, from Lemma~\ref{lem:res} it follows that there exists $b \in \mathbb{Z}$ such that $p \mid f(-b)$ and $\ord_p(b) = m$.
Since $f$ is quadratic and has a root modulo $p$, we have that $f$ splits completely modulo $p$.
Let $a \in \mathbb{Z}$ be such that $f(X) \equiv (X - a)(X + b) \pmod p$.
Recalling that $\alpha\beta = -1$, we get that $ab \equiv 1 \pmod p$, and consequently $\ord_p(a) = \ord_p(b) = m$.
Thus~\ref{ite:a} follows.

Now assume that $\alpha\beta = +1$.
Note that $\gamma\delta = 1$ and $(\gamma + \delta)^2 = \alpha + \beta + 2$ are nonzero coprime integers and $\gamma / \delta = \alpha$ is not a root of unity, so that $(u_k(\gamma, \delta))_{k \geq 0}$ is a Lehmer sequence.
Since $\gamma / \delta = \alpha$ and $\delta / \gamma = \beta$, we have that
\begin{align}\label{equ:R}
\Res(f, \Phi_m) &= \Phi_m(\alpha) \Phi_m(\beta) = \Phi_m(\gamma / \delta) \Phi_m(\delta / \gamma)
 = \Phi_m(\gamma, \delta) \delta^{-\varphi(m)} \Phi_m(\delta, \gamma) \gamma^{-\varphi(m)} \\
    &= \Phi_m(\gamma, \delta) \Phi_m(\delta, \gamma) = \pm \Phi_m(\gamma, \delta)^2 = \pm \Phi_n(\gamma, \delta)^2 . \nonumber
\end{align}

Suppose that~\ref{ite:a} holds.
From Lemma~\ref{lem:res} and~\eqref{equ:R}, it follows that $p \mid \Phi_n(\gamma, \delta)$ and $p \equiv 1 \pmod m$.
Hence, Lemma~\ref{lem:lehmer}\ref{ite:prim_equiv} yields that~\ref{ite:prim} holds.

Now suppose that~\ref{ite:prim} holds.
Then Lemma~\ref{lem:lehmer}\ref{ite:prim_equiv} and~\eqref{equ:R} give that $p \mid \Res(f, \Phi_m)$.
Consequently, by Lemma~\ref{lem:res}, we get that~\ref{ite:a} holds.

The proof of the equivalence of~\ref{ite:a} and~\ref{ite:prim} is complete.
Let us prove that each of~\ref{ite:suff1}, \ref{ite:suff2}, and~\ref{ite:suff3} implies~\ref{ite:prim} (and consequently also~\ref{ite:a}).

Suppose that~\ref{ite:suff1} holds.
Since $4 \nmid m$, $m \notin \{3, 6\}$, and $u_1(\gamma, \delta) = u_2(\gamma, \delta) = 1$, we have that $n \geq 5$ and $2 \nmid n$.
Hence, by Lemma~\ref{lem:lehmer}\ref{ite:prim_odd} it follows that $p \equiv \Big(\tfrac{\gamma\delta(\gamma - \delta)^2}{p}\Big) \equiv 1 \pmod {2n}$, since $\gamma\delta = 1$ and $\gamma - \delta = \alpha + \beta$ is an integer.
Then from $m \mid 2n$ we get that $p \equiv 1 \pmod m$, and~\ref{ite:prim} follows.

Suppose that~\ref{ite:suff2} holds. 
We have that $4 \mid n$, $\gamma\delta = 1$, $\gamma - \delta = \alpha + \beta \in \mathbb{Z}$, $p$ is a primitive divisor of $u_n(\gamma, \delta)$, and $p \equiv 1 \pmod 4$.
From Lemma~\ref{lem:lehmer}\ref{ite:prim_4} it follows that $p \equiv 1 \pmod {2n}$, i.e., $p \equiv 1 \pmod m$, and~\ref{ite:prim} follows.

Suppose that~\ref{ite:suff3} holds.
Then by Lemma~\ref{lem:lehmer}\ref{ite:prim_equiv}, we have that either $p \equiv 1 \pmod m$ or $p \equiv -1 \pmod m$.
Since $4 \mid m$ and $p \equiv 1 \pmod 4$, the second case is impossible.
Therefore, $p \equiv 1 \pmod m$ and~\ref{ite:prim} follows.

The proof is complete.

\section{Proof of Theorem~\ref{thm:odd}}

Let $f, \alpha, \beta, \delta, \gamma$ be as in Theorem~\ref{thm:alphabeta} with $\alpha\beta = -1$, let $g \in \mathbb{Z}[X]$ be a monic polynomial with $f \mid g$, and let $m \geq 7$ be an integer with $m \neq 10$ and $4 \nmid m$.
Since $\gamma\delta = 1$ and $m / (m, 2) \notin \{1, 2, 3, 4, 5, 6, 10, 12\}$, from Lemma~\ref{lem:gammadelta1} we get that $u_{m / (m, 2)}(\gamma, \delta)$ has a primitive divisor $p$.
Hence, from the implication \ref{ite:suff1}$\Rightarrow$\ref{ite:a} of Theorem~\ref{thm:alphabeta} and from Lemma~\ref{lem:criterion}, it follows that $m \in \mathfrak{M}(g)$.
The proof is complete.

\section{Proof of Theorem~\ref{thm:specialD0}}

Let $f, \alpha, \beta, \delta, \gamma$ be as in Theorem~\ref{thm:alphabeta} and let $g \in \mathbb{Z}[X]$ be a monic polynomial with $f \mid g$.
Also, write $(\alpha - \beta)^2 = D_0 D_1^2$, where $D_0, D_1 \in \mathbb{Z}$ and $D_0$ is squarefree, and suppose that $D_0 \geq 5$ and $D_0 \equiv 1 \pmod 4$.

First, assume that $\alpha\beta = -1$.
Hence, we have that $(\gamma^2 - \delta^2)^2 = (\alpha^2 - \beta^2)^2 = D_0 (D_1(\alpha + \beta))^2$, where $D_1(\alpha + \beta)$ is an integer.
Let $m$ be a positive integer with $8D_0 \mid m$.
Since $m / 2 \geq 20$, from Lemma~\ref{lem:gammadelta1} and Lemma~\ref{lem:lehmer}\ref{ite:p2p} and~\ref{ite:r2}, it follows that $u_{m/2}(\gamma, \delta)$ has an odd primitive divisor $p$.
Furthermore, Lemma~\ref{lem:1mod4} yields that $p \equiv 1 \pmod 4$.
Hence,~\ref{ite:suff2} holds and, by Theorem~\ref{thm:alphabeta} and Lemma~\ref{lem:criterion}, we get that $m \in \mathfrak{M}(g)$.

Now assume that $\alpha\beta = +1$.
Hence, we have that $(\gamma^2 - \delta^2)^2 = (\alpha - \beta)^2 = D_0 D_1^2$.
Let $m$ be a positive integer with $4D_0 \mid m$.
Since $m \geq 20$, from Lemma~\ref{lem:gammadelta1} and Lemma~\ref{lem:lehmer}\ref{ite:p2p} and~\ref{ite:r2}, it follows that $u_{m}(\gamma, \delta)$ has an odd primitive divisor $p$.
Furthermore, Lemma~\ref{lem:1mod4} yields that $p \equiv 1 \pmod 4$.
Hence,~\ref{ite:suff3} holds and, by Theorem~\ref{thm:alphabeta} and Lemma~\ref{lem:criterion}, we get that $m \in \mathfrak{M}(g)$.

The proof is complete.

\section{Further remarks}

For the sake of completeness, we also include the case in which $g$ has a linear factor.

\begin{proposition}
Let $g \in \mathbb{Z}[X]$ be a nonconstant monic polynomial with an integer root $a \notin \{-1, 0, +1\}$.
Then every positive integer $m$ belongs to $\mathfrak{M}(g)$, with the possible exception of $m = 2$ if $a = \pm 2^v - 1$ for some positive integer $v$, $m = 3$ if $a = -2$, and $m = 6$ if $a = 2$.
\end{proposition}
\begin{proof}
It is clear that $1 \in \mathfrak{M}(g)$.
By Zsigmondy's theorem~\cite[p.~1]{MR143728}, for every integer $a \notin \{-1, 0, +1\}$ and for every positive integer $m$ with $(a, m) \notin \{(\pm 2^v - 1, 2) : v \geq 1\} \cup \{(-2,3), (2, 6)\}$, there exists a prime number $p$ such that $\ord_p(a) = m$.
Hence, by Lemma~\ref{lem:criterion} with $f(X) = X - a$, we get that $m \in \mathfrak{M}(g)$.
\end{proof}

\bibliographystyle{amsplain}
\bibliography{temp}

\end{document}